\documentclass[]{interact}

\usepackage[caption=false]{subfig}

\usepackage{amsmath}
\usepackage{amsfonts}
\usepackage{euscript}
\usepackage{oldgerm}
\usepackage{rotating}
\usepackage{mathrsfs}
\usepackage{hyperref}
\usepackage{amssymb}
\usepackage{epsfig}
\usepackage{dsfont}
\usepackage{subfig}
\usepackage{xcolor}
\usepackage{graphicx}
\usepackage{subfig}

\renewcommand{\arraystretch}{1.32}

\usepackage[numbers,sort&compress]{natbib}
\bibpunct[, ]{[}{]}{,}{n}{,}{,}

\theoremstyle{plain}
\newtheorem{theorem}{Theorem}[section]
\newtheorem{lemma}[theorem]{Lemma}

\newtheorem{proposition}[theorem]{Proposition}

\theoremstyle{definition}

\theoremstyle{remark}

\newcommand{\norm}[1]{\left\lVert #1\right\rVert}

\begin{document}


\title{A note on superconvergence in projection-based numerical approximations of eigenvalue problems for Fredholm integral operators}

\author{
\name{Shashank K. Shukla\textsuperscript{a,b} \thanks{CONTACT Shashank K. Shukla. Email: shashankshukla.maths@gmail.com; shashankshukla@iitgoa.ac.in}}
\affil{\textsuperscript{a}{School of Mathematics and Computer Science, Indian Institute of Technology Goa, Ponda 403401, Goa, India};\\ \textsuperscript{b}Department of Mathematical Sciences, Rajiv Gandhi Institute of Petroleum Technology, Jais, Amethi 229304, Uttar Pradesh, India.}}

\maketitle

\begin{abstract}
 This paper studies the eigenvalue problem $\mathcal{K}\psi = \lambda \psi$ associated with a Fredholm integral operator $\mathcal{K}$ defined by a smooth kernel. 
 The focus is on analyzing the convergence behaviour of numerical approximations to eigenvalues and their corresponding spectral subspaces. The interpolatory projection methods are employed on spaces of piecewise polynomials of even degree, using $2r+1$ collocation points that are not restricted to Gauss nodes. Explicit convergence rates are established, and the modified collocation method attains faster convergence of approximation of eigenvalues and associated eigenfunctions than the classical collocation scheme. Moreover, it is shown that the iteration yields superconvergent approximations of eigenfunctions. Numerical experiments are presented to validate the theoretical findings.
\end{abstract}

\begin{keywords}
Eigenvalue problems, Fredholm integral operators,  Interpolatory projections, Collocation points, Spectral subspace
\end{keywords}

\begin{amscode}
45C05; 47A75; 65R15; 65R20
\end{amscode}

\setcounter{equation}{0}
\section{Introduction}


Eigenvalue problems for Fredholm integral operators appear in numerous applications in applied mathematics, physics, and engineering; see, e.g., \cite{Chatelin, Chen1997}. 
Because exact eigenvalues are rarely available in closed form, the construction of effective numerical methods is essential. Extensive research has therefore been devoted to the approximation of eigenvalues and eigenfunctions of compact operators; see \cite{Ahues, Atkinson1967, Atkinson1975, Babuska1987, Babuska1989, Bramble1973, Osborn, Sloan1}. 
A common strategy is to discretize the operator to obtain a finite-dimensional matrix eigenvalue problem whose solutions approximate the eigenpairs of the original operator.
Widely used discretization techniques include projection methods, Nystr\"om schemes, and degenerate kernel approaches; see \cite{KEA0, KEA-IG-IS, Cha-Leb2, Gnan1, Pallav2002}. Their convergence properties and error behavior have been thoroughly investigated in the literature; see \cite{Chatelin, Kato1976, Osborn, Atkinson1975}. Since such discretizations typically lead to matrix eigenvalue systems, computational efficiency becomes a critical concern, motivating the continued development of accurate and efficient numerical algorithms.

Let $\mathcal{X} = \mathcal{C}[0,1]$ denote the Banach space of real-valued continuous functions defined on the interval $[0,1]$, endowed with the supremum norm
$\norm{x}_\infty = \sup_{t \in [0,1]} |x(t)|$, and let $\Omega = [0,1]\times[0,1]$.  
Suppose that $\kappa : \Omega \to \mathbb{R}$ is a smooth function. Define the Fredholm (linear) integral operator $\mathcal{K} : \mathcal{X} \to \mathcal{X}$ by
\begin{equation} \label{Eq:01}
	(\mathcal{K}x)(s)
	=
	\int_0^1 \kappa(s,t)\,x(t)\,dt,
	\quad s \in [0,1].
\end{equation}
Under the smoothness assumption on $\kappa$, the operator $\mathcal{K}$ is bounded and compact on $\mathcal{X}$. The principal problem considered in this work is the eigenvalue problem
\begin{equation} \label{Eq:02}
	\mathcal{K}\psi = \lambda \psi,
	\quad 
	\lambda \in \mathbb{C}\setminus\{0\},
	\quad 
	\psi \in \mathcal{X}\setminus\{0\}.
\end{equation}
Since $\mathcal{K}$ is compact, its spectrum consists of at most a countable set of nonzero eigenvalues with finite algebraic multiplicity, and zero is the only possible accumulation point. Moreover, eigenfunctions corresponding to nonzero eigenvalues belong to $\mathcal{X}$ and inherit additional smoothness when the kernel $\kappa$ is smooth.

The numerical approximation of \eqref{Eq:02} is formulated within the framework of spectral approximation theory for compact operators. Let $\{\mathcal{K}_n\}$ be a sequence of continuous finite-rank operators converging pointwise to $\mathcal{K}$. We assume that the family $\{\mathcal{K}_n\}$ is collectively compact, that is,
$
\left\{
\mathcal{K}_n x 
:\ 
\ \norm{x}_\infty \le 1, n \in \mathbb{N}
\right\}
$
has compact closure in $\mathcal{X}$.  
The eigenvalue problem \eqref{Eq:02} is approximated by  
\begin{equation} \label{Eq:03}
	\mathcal{K}_n \psi_n = \lambda_n \psi_n,
	\quad 
	\lambda_n \in \mathbb{C},
	\quad 
	\psi_n \in \mathcal{X}\setminus\{0\}.
\end{equation}
Since each $\mathcal{K}_n$ has finite rank, \eqref{Eq:03} reduces to a matrix eigenvalue problem; see \cite{Ahues, KEA0}. Classical results in compact operator approximation theory guarantee that isolated eigenvalues of $\mathcal{K}$ are approximated by eigenvalues of $\mathcal{K}_n$, and the corresponding invariant subspaces converge in the gap metric. Also, the approximated eigenpairs $(\lambda_n, \psi_n)$ converge to the corresponding exact eigenpair $(\lambda, \psi)$, , and precise error estimates can be derived.

The compactness of integral operators implies a discrete spectral structure, which makes projection-based methods particularly effective for eigenvalue approximation. 
For operators with smooth kernels, the spectral convergence of Galerkin and collocation methods has been extensively analyzed; see \cite{Ahues, KEA0, Baker0, Chatelin, Sloan0}. 
One-step iteration enhancement of projection schemes was first proposed by Sloan \cite{Sloan1}, and later refined through modified projection methods \cite{RPK5, RPK6}, where improved convergence rates for eigenvalues and eigenfunctions were established. 
Further acceleration was obtained by applying a one-step iteration to the modified projection schemes.

Alternative discretization strategies include Nystr\"om methods, which replace the integral by accurate quadrature rules, and degenerate kernel methods, which approximate the kernel by separable expansions; see \cite{Gnan1}. For smooth kernels, these approaches also provide high-order convergence, and superconvergent modified or iterated variants have been reported in \cite{CA-PS-DS-MT}. In most existing frameworks, for interpolation the Gaussian points (zeros of Legendre polynomials) are selected as interpolation nodes. In contrast, the present work employs an interpolatory projection by selecting equidistant collocation points that does not rely on the zeros of any special functions and leading to a simpler and more flexible implementation.

In this work, we consider interpolatory projection methods for constructing finite-rank approximations $\mathcal{K}_n$ of the compact operator $\mathcal{K}$. 
On a uniform partition of $[0,1]$ with mesh size $h$, the approximating space $\mathcal{X}_n$ consists of piecewise polynomials of even degree at most $2r$. 
The projection operator $\mathcal{Q}_n : \mathcal{X} \to \mathcal{X}_n$ is defined using $2r+1$ equidistant collocation points in each subinterval, without requiring Gauss nodes. 
For smooth functions, $\mathcal{Q}_n$ converges uniformly to the identity and forms a uniformly bounded family. The operator $\mathcal{K}$ is approximated by classical collocation-type operators $\mathcal{K}_n = \mathcal{Q}_n \mathcal{K}$ and by modified collocation operators 
$
\mathcal{K}_n^{M}
=
\mathcal{Q}_n \mathcal{K}
+
\mathcal{K}\mathcal{Q}_n
-
\mathcal{Q}_n \mathcal{K}\mathcal{Q}_n,
$
which enhance eigenvalue convergence while retaining finite rank. We establish convergence rates for eigenvalues and spectral subspaces and show that, for smooth kernels, high-order convergence is achieved with such equidistant collocation points.  Moreover, iteration of the modified method yields superconvergent approximations of associated eigenfunctions. The analysis is carried out in the uniform norm. These results show that high-order convergence can be achieved within our framework without the need to employ collocation points derived from special functions.

The paper is organized as follows. Section~2 introduces the interpolatory projection and the associated collocation-based schemes for eigenvalue problems. Section~3 presents the spectral projection framework and the theoretical results required for the analysis of eigenvalue approximations, including the relevant error estimates. In Section~4, we establish the main convergence results for the classical and modified collocation methods, together with their iterated versions. Section~5 contains numerical experiments that illustrate and confirm the theoretical findings. Finally, Section~6 concludes the paper.

\section{Projection-based Approximations}

Denote the uniform partition of $[0,1]$ as
$$
0 =t_0 < t_1 < \cdots <t_n=1, \quad \text{where} ~ t_j = \displaystyle{\frac{j}{n}}, \quad j=0,1, \ldots, n.
$$ 
Denote $\Delta_{j} = [t_{j-1}, t_j]$. Let $r \ge 0$ be an integer, and define
$$
\mathcal{X}_n = \left\{x \in \mathcal{L}^\infty[0, 1] : x|_{\Delta_{j}} \text{ is a polynomial of even degree} \leq 2r, \; j=1,2, \ldots, n \right\}. 
$$ 

\noindent
The space $\mathcal{X}_n$ is finite dimensional with $\dim(\mathcal{X}_n) = n(2r+1)$. 
As no continuity is imposed at the partition points, $\mathcal{X}_n \subset \mathcal{L}^\infty[0,1]$.
We consider the interpolatory projection $\mathcal{Q}_n: \mathcal{X} \to \mathcal{X}_n$, 
defined by interpolation at $2r+1$ points on each subinterval of the partition. On every subinterval, $\mathcal{Q}_nx$ is the polynomial of degree at most $2r$ that agrees with $x$ at the chosen nodes, so that $\mathcal{Q}_nx \in \mathcal{X}_n$ and matches $x$ at all interpolation points and will be discussed in the subsequent subsection.

\subsection{Interpolatory Projection}
On each subinterval $[t_{j-1}, t_j]$ of a uniform partition with mesh size $h$, we select $2r+1$ equidistant interpolation points as 
\begin{equation*}
	\tau_j^i = t_{j-1} + \frac{ih}{2r}, \quad i = 0, 1, \ldots, 2r  \; \text{ for } \; r \geq 1, \quad \text{and} \quad \tau_j = \frac{t_{j-1} + t_j}{2} \; \text{ for }  \; r = 0.
\end{equation*}
Define the interpolatory operator $\mathcal{Q}_n : \mathcal{X} \to \mathcal{X}_n$ by interpolation at these nodes given by
\begin{equation} \label{Eq:04}
	\mathcal{Q}_nx(\tau_j^i) = x(\tau_j^i), \quad i = 0, 1, 2, \ldots, 2r; ~ j = 1, 2, \ldots, n.
\end{equation}
Then $\mathcal{Q}_n x \to x$ uniformly for all $x \in \mathcal{C}[0,1]$.  By employing the Hahn-Banach extension theorem, \(\mathcal{Q}_n\) can be extended to \(\mathcal{L}^\infty[0,1]\), making \(\mathcal{Q}_n : \mathcal{L}^\infty[0,1] \to \mathcal{X}_n\) a projection. Define the polynomial 
\begin{eqnarray*}
	\Phi_j(t) & = & \prod_{i=0}^{2r} \left(t - \tau_j^i \right), \quad j = 1, 2, \ldots, n.
\end{eqnarray*}
As \(\mathcal{Q}_{n,j}x\) (where \(\mathcal{Q}_{n,j}\) is \(Q_{n}\) restricted on \(\Delta_j\)) interpolates \(x\) at the points \(\tau_j^0, \tau_j^1, \ldots, \tau_j^{2r}\), we have
\begin{eqnarray*}
	x(t) - \mathcal{Q}_{n,j}x(t) & = & \Phi_j(t) \left[\tau_j^0, \tau_j^1, \ldots, \tau_j^{2r}, t\right]x, \quad t \in \Delta_j,
\end{eqnarray*}
where $\left[\tau_j^0, \tau_j^1, \ldots, \tau_j^{2r}, t\right]x$ is the $(2r+1)$-th divided difference of $x \in \mathcal{C}[0, 1]$ at $\tau_j^0, \ldots, \tau_j^{2r}$ defined as
\begin{eqnarray*}
	\left[\tau_j^0, \ldots, \tau_j^{2r}\right]x &=& \frac{\left[\tau_j^1, \ldots, \tau_j^{2r}\right]x - \left[\tau_j^0, \ldots, \tau_j^{2r-1}\right]x}{\tau_j^{2r} - \tau_j^0}.
\end{eqnarray*}
If \(x \in\mathcal{C}^{2r+1}(\Delta_j)\), a standard result implies that
\begin{equation*} 
	\norm{(\mathcal{I} - \mathcal{Q}_{n,j})x}_{\Delta_j, \infty}  \leq C_1 \norm{x^{(2r+1)}}_\infty \, h^{2r+1},
\end{equation*}
where $C_1$ is a constant independent of $h$. In general, for $x \in \mathcal{C}^{\alpha}[0,1]$, where $\beta = \min\{\alpha, 2r+1\}$, one can see that
\begin{equation}  \label{Eq:05}
	\norm{(\mathcal{I}-\mathcal{Q}_n)x}_\infty 
	\le C_1 \norm{x^{(\beta)}}_\infty \, h^{\beta}.
\end{equation}

\subsection{Approximation Methods for Interpolatory Projection}

\noindent
The classical collocation method for approximating the eigenvalue problem 
\eqref{Eq:02} is obtained by replacing the operator $\mathcal{K}_n$ in 
\eqref{Eq:03} with $\mathcal{Q}_n \mathcal{K}$. 
Thus, we seek $\lambda_n^C \in \mathbb{C}\setminus\{0\}$ and 
$\psi_n^C \in \mathcal{X}_n$ with $\|\psi_n^C\|_\infty = 1$ such that
\[
\mathcal{Q}_n \mathcal{K} \psi_n^C = \lambda_n^C \psi_n^C .
\]
To improve the eigenfunction approximation, Sloan \cite{Sloan1} proposed a 
one-step iteration defined by
\[
\psi_n^S = \frac{1}{\lambda_n^C}\mathcal{K}\psi_n^C ,
\]
which satisfies
\[
\mathcal{K}\mathcal{Q}_n \psi_n^S = \lambda_n^C \psi_n^S .
\]

\medskip
\noindent
In the modified collocation method introduced by Kulkarni \cite{RPK6}, 
$\mathcal{K}_n$ is replaced by the finite-rank operator
\[
\mathcal{K}_n^M 
= \mathcal{Q}_n \mathcal{K} \mathcal{Q}_n
+ \mathcal{Q}_n \mathcal{K}(\mathcal{I}-\mathcal{Q}_n)
+ (\mathcal{I}-\mathcal{Q}_n)\mathcal{K}\mathcal{Q}_n
= \mathcal{Q}_n \mathcal{K} + \mathcal{K}\mathcal{Q}_n 
- \mathcal{Q}_n \mathcal{K} \mathcal{Q}_n .
\]
Consequently,
\[
\|\mathcal{K}-\mathcal{K}_n^M\|
= \|(\mathcal{I}-\mathcal{Q}_n)\mathcal{K}(\mathcal{I}-\mathcal{Q}_n)\|
\to 0 \quad \text{as } n \to \infty.
\]
The modified eigenvalue problem then consists in finding 
$\lambda_n^M \in \mathbb{C}\setminus\{0\}$ and 
$\psi_n^M \in \mathcal{X}$ with $\|\psi_n^M\|_\infty = 1$ such that
\[
\mathcal{K}_n^M \psi_n^M = \lambda_n^M \psi_n^M .
\]
The corresponding iterated modified collocation eigenfunction is defined by
\[
\widetilde{\psi}_n^M 
= \frac{1}{\lambda_n^M}\mathcal{K}\psi_n^M .
\]

\section{Spectral Projection and Error Estimation} \label{ep_spectral_projection}

The convergence analysis is carried out using tools from operator theory, notably spectral projections associated with isolated eigenvalues. This framework allows a clear understanding of how approximate eigenvalues and eigenspaces converge to the exact ones and forms the basis for the convergence estimates and numerical results presented in the subsequent sections.

Let $\mathcal{X} = \mathcal{C}[0,1]$, and $\mathcal{BL}(\mathcal{X})$ will denote the set of all bounded linear operators on $\mathcal{X}$, equipped with the operator norm. Let  $\mathcal{K} : \mathcal{X} \to \mathcal{X}$ be a compact linear operator, and let the resolvent set of $\mathcal{K}$ is defined as 
\[
\rho(\mathcal{K}) = \{ z \in \mathbb{C} : (\mathcal{K} - z\mathcal{I})^{-1} \in \mathcal{BL}(\mathcal{X}) \}
\]
and the spectrum of $\mathcal{K}$ is $\sigma(\mathcal{K}) = \mathbb{C} \setminus \rho(\mathcal{K})$. The point spectrum of $\mathcal{K}$ contains all $\lambda \in \sigma(\mathcal{K})$ for which $\mathcal{K} - \lambda \mathcal{I}$ is not injective. In such cases, $\lambda$ is referred as an eigenvalue of $\mathcal{K}$. The geometric multiplicity of $\lambda$ is defined as the dimension of the null space of $\mathcal{K} - \lambda \mathcal{I}$. Let $\Gamma$ be a rectifiable, simple, closed, positively oriented curve contained in $\rho(\mathcal{K})$ that separates $\lambda$ from the remainder of the spectrum of $\mathcal{K}$. The spectral projection corresponing to  $\mathcal{K}$ and $\lambda$ is then given by:
\[
E = -\frac{1}{2\pi i} \int_{\Gamma} (\mathcal{K} - z\mathcal{I})^{-1} \, dz \in  \mathcal{BL}(\mathcal{X}),
\]
and the dimension of $\mathcal{R}(E)$, the range of the operator $E$, is referred as the algebraic multiplicity of $\lambda$. If $\lambda$ has finite algebraic multiplicity, it is classified as a spectral value of finite type. In this case, $\lambda$ is an eigenvalue of $\mathcal{K}$. ( See Ahues et al \cite[Proposition 1.31]{Ahues}.) Furthermore, if the algebraic multiplicity of $\lambda$ is equal to one, $\lambda$ is called a simple eigenvalue of $\mathcal{K}$. For any nonzero subspaces $\mathcal{Y}$ and $\mathcal{Z}$ of $\mathcal{X}$, define 
\[
\delta(\mathcal{Y}, \mathcal{Z}) = \sup \{d(y, \mathcal{Z}) : y \in \mathcal{Y}, \norm{y}_\infty = 1\},
\]
where $d(y, \mathcal{Z}) = \inf_{z \in \mathcal{Z}} \norm{y - z}_\infty$. The gap between $\mathcal{Y}$ and $\mathcal{Z}$ is then given by  
\[
\hat{\delta}(\mathcal{Y}, \mathcal{Z}) = \max\{\delta(\mathcal{Y}, \mathcal{Z}), \delta(\mathcal{Z}, \mathcal{Y})\}.
\]	
This notion provides a quantitative measure of the distance between
subspaces and will be used to study convergence of approximate
eigenspaces.

Consider a sequence of operators $\{\mathcal{K}_n\} \subset \mathcal{BL}(\mathcal{X})$ that converges to $\mathcal{K}$ in a collectively compact manner, that is,
$\mathcal{K}_n x \to \mathcal{K}x$ for all $x \in \mathcal{X}$ and the set
$\{\mathcal{K}_{n}x : \norm{x}_\infty \le 1,\ n \in \mathbb{N}\}$. As discussed in Osborn~\cite{Osborn}, for all sufficiently large $n$, the contour $\Gamma$ is contained in the resolvent set $\rho(\mathcal{K}_n)$ and
$\sup_{z \in \Gamma} \norm{(\mathcal{K}_n - z\mathcal{I})^{-1}} < \infty.$
Consequently, the spectral projection associated with $\mathcal{K}_n$, defined by
\[
E_n = -\frac{1}{2\pi i} \int_{\Gamma} (\mathcal{K}_n - z\mathcal{I})^{-1} \, dz,
\]
is well defined and has rank $m$. In particular, the spectrum of $\mathcal{K}_n$ inside $\Gamma$ consists of exactly $m$ eigenvalues $\lambda_{n,1}, \lambda_{n,2}, \ldots, \lambda_{n,m}$, counted according to their algebraic multiplicities.
Define the arithmetic mean of these eigenvalues as  
\[
\hat{\lambda}_n = \frac{1}{m} \sum_{j=1}^m \lambda_{n,j}.
\]	

Now, consider the following fundamental result of Osborn~\cite{Osborn}, which provides an estimate for the gap between the exact and approximate spectral subspaces associated with a compact operator, and hence measures the closeness
of the corresponding eigenfunctions.
\begin{lemma}\label{Lem:01}
	For all sufficiently large $n$, the gap between the spectral subspaces satisfies
	\[
	\hat{\delta}\big(\mathcal{R}(E), \mathcal{R}(E_n)\big)
	\le C \norm{(\mathcal{K}-\mathcal{K}_n)\mathcal{K}|_{\mathcal{R}(E)}},
	\]
	where $C >0$ is a constant independent of $n$, which takes different values at different places, and $(\mathcal{K}-\mathcal{K}_n)\mathcal{K}|_{\mathcal{R}(E)}$ denotes the restriction of $(\mathcal{K}-\mathcal{K}_n)\mathcal{K}$ to the spectral subspace $\mathcal{R}(E)$.
\end{lemma}
Also, we consider a modified result of Osborn~\cite[Theorem~2]{Osborn}, as presented
by Kulkarni~\cite[Theorem~2.2]{RPK6}, which provides improved error bounds for the
approximation of eigenvalues. Specifically, this result yields the following
estimate for the eigenvalue error
\begin{lemma}\label{Lem:02}
	For all sufficiently large $n$, the arithmetic mean $\hat{\lambda}_n$ of the
	approximating eigenvalues satisfies
	\[
	|\lambda-\hat{\lambda}_n|
	\le C \norm{\mathcal{K}_n(\mathcal{K}-\mathcal{K}_n)\mathcal{K}|_{\mathcal{R}(E)}}.
	\]
\end{lemma}

\noindent
The above two theorems establish the essential error bounds for the eigenelements. The following result can be obtained from Lemmas~ \ref{Lem:01} and \ref{Lem:02}.

\begin{theorem}\label{Thm:01}
	Let $\{\mathcal{K}_n\}$ be a sequence of bounded linear operators on $\mathcal{X}$ converging to $\mathcal{K}$ in the collectively compact sense, with
	$\mathcal{K}_n = \mathcal{Q}_n \mathcal{K}$. Llet $E$ and $E_n$ denote the spectral
	projections of $\mathcal{K}$ and $\mathcal{K}_n$, respectively, associated with their
	spectra enclosed by a simple closed contour $\Gamma$.
	Let $\lambda$ be an eigenvalue of $\mathcal{K}$ with algebraic multiplicity $m$, and let
	$\hat{\lambda}_n$ denote the arithmetic mean of the $m$ eigenvalues of $\mathcal{K}_n$ inside $\Gamma$. Denote the corresponding spectral subspaces by $\mathcal{R}(E)$ and $\mathcal{R}(E_n)$. For any $\psi_n^C \in\mathcal{R}(E_n)$, define $\psi_n^{S}=\mathcal{K}\psi_n^C$. Then, for all sufficiently large $n$, there exists a constant $C>0$, independent of $n$, such that
	\begin{equation*} \label{eq: st_roc}
		\norm{\psi_n^C - E \psi_n^C} \leq \hat{\delta} \big(\mathcal{R}(E), \mathcal{R}(E_n)\big) \leq C \norm{(\mathcal{K} - \mathcal{Q}_n \mathcal{K})\mathcal{K}|_{\mathcal{R}(E)}},
	\end{equation*}
	\begin{equation*} \label{eq: it_roc}
		\norm{\psi_n^S - E \psi_n^S} \leq \delta\big(\mathcal{R}(E), \mathcal{K} \mathcal{R}(E_n)\big) \leq C \norm{\mathcal{K}(\mathcal{K} - \mathcal{Q}_n \mathcal{K})\mathcal{K}|_{\mathcal{R}(E)}},
	\end{equation*}
	and 
	\begin{equation*} \label{eq: st_eig}
		| \lambda - \hat{\lambda}_n | \leq C \norm{ \mathcal{K}(\mathcal{K} - \mathcal{Q}_n \mathcal{K})\mathcal{K}|_{\mathcal{R}(E)}}.
	\end{equation*}
\end{theorem}

Since $\mathcal{K}_n^M$ converges to $\mathcal{K}$ in the operator norm, it follows
that for all sufficiently large $n$, the operator $\mathcal{K}_n^M$ has exactly
$m$ eigenvalues
\(
\lambda_{n,1}^M,\; \lambda_{n,2}^M,\; \dots,\; \lambda_{n,m}^M
\)
inside the contour $\Gamma$, counted according to algebraic multiplicity.
Let
\[
\hat{\lambda}_n^M
=
\frac{1}{m}\sum_{j=1}^{m}\lambda_{n,j}^M
\]
denote the arithmetic mean of these eigenvalues, and let $E_n^M$ denote the
corresponding spectral projection. Then, as a consequence of Theorems~\ref{Lem:01} and~\ref{Lem:02}, the following result holds.

\begin{lemma}[Kulkarni~\cite{RPK6}]\label{Lem:03}
	For all sufficiently large $n$, there exists a constant $C>0$, independent of
	$n$, such that
	\begin{equation*}\label{Eq:gap_modified}
		\hat{\delta}\bigl(\mathcal{R}(E),\mathcal{R}(E_n^M)\bigr)
		\le
		C\,
		\norm{(\mathcal{K}-\mathcal{K}_n^M)\mathcal{K}
			\big|_{\mathcal{R}(E)}},
	\end{equation*}
	and
	\begin{equation*}\label{Eq:eigenvalue_modified}
		|{\lambda - \hat{\lambda}_n^M}|
		\le
		C\,
		\norm{\mathcal{K}(\mathcal{K}-\mathcal{K}_n^M)\mathcal{K}
			\big|_{\mathcal{R}(E)}}.
	\end{equation*}
\end{lemma}

\noindent
The following result can be obtained from the above Lemma \ref{Lem:03}; see Bouda et al.~\cite{HB-CA-ZEA-KK}.
\begin{theorem}\label{Thm:02}
	Let $\{\mathcal{K}_n\}\subset \mathcal{BL}(\mathcal{X})$ converging to $\mathcal{K}$ in the collectively compact sense, with $\mathcal{K}_n=\mathcal{K}_n^M$. Let $E$ and $E_n^M$ denote the spectral
	projections of $\mathcal{K}$ and $\mathcal{K}_n^M$, respectively, associated with their spectra enclosed by a simple closed contour $\Gamma$. Let $\lambda$ be an eigenvalue of $\mathcal{K}$ with algebraic multiplicity $m$, and let $\hat{\lambda}_n^M$ denote the arithmetic mean of the $m$ eigenvalues of $\mathcal{K}_n^M$ inside $\Gamma$. Denote the corresponding spectral subspaces by $\mathcal{R}(E)$ and $\mathcal{R}(E_n^M)$. For any $\psi_n^M \in \mathcal{R}(E_n^M)$, $\widetilde{\psi}_n^M = \mathcal{K} \psi_n^M$. Then, for all sufficiently large $n$, there exists a constant $C>0$, independent of $n$, such that
	\begin{equation*} \label{eq: mod_roc}
		\norm{\psi_n^M - E \psi_n^M} \leq \hat{\delta} \big(\mathcal{R}(E), \mathcal{R}(E_n^M)\big) \leq C \norm{(\mathcal{K} - \mathcal{K}_n^M)\mathcal{K}|_{\mathcal{R}(E)}},
	\end{equation*}
	\begin{equation*} \label{eq: itmod_roc}
		\norm{\widetilde{\psi}_n^M - E \widetilde{\psi}_n^M} \leq \delta\big(\mathcal{R}(E), \mathcal{K} \mathcal{R}(E_n^M)\big) \leq C \norm{\mathcal{K}(\mathcal{K} - \mathcal{K}_n^M)\mathcal{K}|_{\mathcal{R}(E)}},
	\end{equation*}
	and 
	\begin{equation*} \label{eq: mod_eig}
		| \lambda - \hat{\lambda}_n^M | \leq C \norm{ \mathcal{K}(\mathcal{K} -  \mathcal{K}_n^M)\mathcal{K}|_{\mathcal{R}(E)}}.
	\end{equation*}
\end{theorem}

\noindent
Theorems~\ref{Thm:01} and~\ref{Thm:02} present the principal results of this work. Theorem~\ref{Thm:01} establishes the convergence of the approximate eigenspaces $\mathcal{R}(E_n)$ to the exact eigenspace $\mathcal{R}(E)$ for the collocation method, together with error estimates for the associated eigenfunctions and for the arithmetic mean of the approximate eigenvalues. Theorem~\ref{Thm:02} extends this analysis to the modified collocation method and demonstrates improved error bounds, thereby confirming the higher accuracy of the modified scheme. The similar results


Throughout this work, we restrict our attention to the case where 
$\lambda$ is a simple eigenvalue of $\mathcal{K}$. 
Accordingly, we denote the collocation and modified collocation 
approximations of $\lambda$ by 
$\hat{\lambda}_n = \lambda_n^{C}$ and 
$\hat{\lambda}_n^{M} = \lambda_n^{M}$, respectively.

\section{Convergence Rates in Approximation Methods}
In this section, we recall the Fredholm integral operator $\mathcal{K}$ defined by \eqref{Eq:01} with a smooth kernel 
$\kappa \in \mathcal{C}^{2r+2}(\Omega)$, given by
\begin{equation*} 
	(\mathcal{K}x)(s) = \int_{0}^{1} \kappa(s, t) x(t)~dt, 
	\quad  s \in [0,1], \quad x \in \mathcal{X}.
\end{equation*}
Then $\mathcal{K} : \mathcal{C}[0,1] \to \mathcal{C}[0,1]$ is a compact linear operator.
We further assume that the integral operator $\mathcal{K}$ admits a continuous
extension to $\mathcal{L}^\infty[0,1]$ and satisfies
$
\mathcal{R}(\mathcal{K}) \subset \mathcal{C}^{2r+2}[0,1].
$
Now, we introduce the notation and several preliminary results that will be utilized in subsequent convergence analysis.  
Let $\alpha$ and $\beta$ be non-negative integers such that $0 \le \alpha + \beta \le 2r + 3$. We define the mixed partial derivative of the kernel $\kappa$ by
\begin{equation*}
	D^{(\alpha,\beta)}\kappa(s,t) = 
	\frac{\partial^{\alpha+\beta}\kappa}{\partial s^\alpha \partial t^\beta}(s,t),
\end{equation*}
and set
\begin{equation*}
	C_2 = \max_{0 \le \alpha + \beta \le 2r+3} 
	\left\{
	\sup_{0 \le s,t \le 1} 
	\left| D^{(\alpha,\beta)}\kappa(s,t) \right|
	\right\}.
\end{equation*}
Here, the constant $C_2$ provides a uniform upper bound for all partial derivatives of $\kappa$ up to total order $2r+3$.

\noindent
For any non-negative integer $\gamma$, if $x \in \mathcal{C}^{2r+\gamma}[0,1]$, we define the corresponding maximum norm as
\begin{equation*}
	\norm{x}_{2r+\gamma, \infty} 
	= \max_{0 \le i \le 2r+\gamma} 
	\norm{x^{(i)}}_\infty
	= \max_{0 \le i \le 2r+\gamma} 
	\left\{ \sup_{t \in [0,1]} \left|x^{(i)}(t)\right| \right\},
\end{equation*}
where $x^{(i)}$ denotes the $i$-th derivative of $x$. 

\noindent
For $x \in \mathcal{C}[0,1]$ and $0 \le j \le 2r+2$, we have
\[
(\mathcal{K}x)^{(j)}(s)
= \int_{0}^{1} \frac{\partial^{j} \kappa}{\partial s^j }(s,t)\, x(t)\, dt = \int_0^1 D^{(j,0)}\kappa(s,t)\, x(t)\, dt,
\quad s \in [0,1].
\]
Hence,
\begin{equation} \label{Eq:06}
	\norm{(\mathcal{K}x)^{(j)}}_{\infty} \le C_2 \norm{x}_\infty.
\end{equation}

We now establish several propositions that quantify the error behaviour of integral operators involving $\mathcal{K}$ and $\mathcal{Q}_n$, under the assumption that the kernel $\kappa$ is smooth.

\begin{proposition} \label{Pr:01}
	If the kernel  $\kappa$ is continuously differentiable with respect to $t$ 
	and \(x \in \mathcal{C}^{2r+2}[0,1]\), then
	\[
	\norm{\mathcal{K}(\mathcal{I} - \mathcal{Q}_n)x }_\infty 
	\le 2C_2 \norm{x}_{2r+2, \infty} h^{2r + 2}.
	\]
\end{proposition}
\vspace*{-0.2in}
\begin{proof} For \(s \in [0,1]\), we write
	\begin{align*}
		\mathcal{K}(\mathcal{I} - \mathcal{Q}_n)x(s) &=\int_0^1  \kappa(s,t) (\mathcal{I} - \mathcal{Q}_{n}) x(t) \, dt \notag \\ 
		&= \sum_{j=1}^{n} \int_{t_{j-1}}^{t_j}  \kappa(s,t) (\mathcal{I} - \mathcal{Q}_{n,j})x(t) \, dt. \notag 
	\end{align*}
	For each subinterval $[t_{j-1},t_j]$, the interpolation remainder can be expressed as
	\[
	(\mathcal{I}-\mathcal{Q}_{n,j})x(t)
	= \Phi_j(t)\,[\tau_j^0,\tau_j^1,\ldots,\tau_j^{2r},t]x,
	\quad t\in[t_{j-1},t_j],
	\]
	where $	\Phi_j(t) = \prod_{i=0}^{2r} \bigl(t - \tau_j^i\bigr) $.
	Then,
	\[
	\mathcal{K}(\mathcal{I} - \mathcal{Q}_n)x(s) = \sum_{j=1}^{n ~}\int_{t_{j-1}}^{t_j}  \kappa(s,t) \Phi_j(t)  \left[\tau_j^0,\tau_j^1,\ldots, \tau_j^{2r},t\right]x \, dt,\]
	For fixed \(s\), let
	\begin{align*}
		f_j^s(t) =  \kappa(s,t) \left[\tau_j^0,\tau_j^1,\ldots, \tau_j^{2r},t\right]x, \quad t \in [t_{j-1}, t_j], ~ \text{for} ~ j=1,2,\ldots,n.
	\end{align*}
	Then
	\begin{equation} \label{Eq:07}
		\mathcal{K}(\mathcal{I} - \mathcal{Q}_n)x(s)= \sum_{j=1}^{n ~}\int_{t_{j-1}}^{t_j} f_j^s(t) \Phi_j(t) \, dt
	\end{equation}
	For $j = 1,2,\ldots,n$, define
	\begin{equation*}
		\omega_j(t) = \int_{t_{j-1}}^t \Phi_j(s) \, ds.
	\end{equation*}
	Then, by Leibnitz rule
	\begin{equation*}
		\omega_j'(t) =  \Phi_j(t), ~  \omega_j(t) \ge 0 ~~ \text{for}~ t \in [t_{j-1},t_j],
	\end{equation*}
	Obseve that, $	\omega_j(t_{j-1})=0$ and it is easy to see that ~$\omega_j(t_{j})=0$.\\
	Using above result in equation \eqref{Eq:07}, we have
	\begin{align} \label{Eq:08}
		\mathcal{K}(\mathcal{I} - \mathcal{Q}_n)x(s) &= \sum_{j=1}^{n ~}\int_{t_{j-1}}^{t_j} f_j^s(t) \omega_j'(t) \, dt \notag \\
		&=-\sum_{j=1}^{n ~}\int_{t_{j-1}}^{t_j} (f_j^s)'(t) \omega_j(t) \, dt.
	\end{align}
	Note that, 	
	\begin{equation*}
		(f_j^s)'(t) ~= ~
		\kappa(s,t) \left[\tau_j^0,\tau_j^1,\ldots, \tau_j^{2r},t, t \right]x \\ 
		+ D^{(0,1)} 	\kappa(s,t) \left[\tau_j^0,\tau_j^1,\ldots, \tau_j^{2r}, t \right]x.
	\end{equation*}
	From Hilderbrand \cite{Hild}, we recall standard result of divided difference as follows:\\
	Since $x \in \mathcal{C}^{2r+2}[0,1]$ and $t \in [t_{j-1},t_j]$, we have
	\begin{align*}
		\left[\tau_j^0,\tau_j^1,\ldots, \tau_j^{2r},t \right]x = \frac{x^{(2r+1)}(\xi_1^j)}{(2r+1)!}, \quad \text{and} \quad
		\left[\tau_j^0,\tau_j^1,\ldots, \tau_j^{2r},t, t \right]x = \frac{x^{(2r+2)}(\xi_2^j)}{(2r+2)!},
	\end{align*}
	for some $\xi_1^j, \xi_2^j \in [t_{j-1},t_j]$. Therefore, for $s \in [0,1]$ and $t \in [t_{j-1},t_j]$,
	\begin{equation*}
		\left|(f_j^s)'(t) \right| \le 2 C_2 \norm{x}_{2r +2, \infty}.
	\end{equation*}
	Also, for $j=1,2,\ldots,n$, we observe that
	\begin{equation*}
		\norm{\Phi_j}_\infty = \max_{t \in [0,1]} |\Phi_j(t)| = O\left( h^{2r+1}\right).
	\end{equation*}
	Using \eqref{Eq:08}, we obtain
	\begin{align*}
		\left|\mathcal{K}(\mathcal{I} - \mathcal{Q}_n)x(s)\right| &\le \sum_{j=1}^{n} \int_{t_{j-1}}^{t_j} \int_{t_{j-1}}^t |(f_j^s)'(t)| |\Phi_j(s)| \,ds \,dt \\
		& \le 2 C_2 \norm{x}_{2r +2, \infty} h^{2r + 2}.
	\end{align*}
	Thus, we have
	\begin{equation*}
		\norm{\mathcal{K}(\mathcal{I} - \mathcal{Q}_n)x}_\infty \le 2 C_2 \norm{x}_{2r +2, \infty} h^{2r + 2},
	\end{equation*}
	which completes the proof.
\end{proof}

\begin{proposition} \label{Pr:02}
	If the kernel $\kappa \in \mathcal{C}^{2r+1}(\Omega)$
	and $x \in \mathcal{C}^{2r+2}[0,1]$, then
	\[
	\norm{(\mathcal{I}-\mathcal{Q}_n)\mathcal{K}(\mathcal{I}-\mathcal{Q}_n)x}_\infty
	\le C_1C_3 \norm{x}_{2r+2,\infty}\,h^{4r+3}
	\]
	for some constant $C_3$ independent of $h$.
\end{proposition}
\vspace*{-0.2in}
\begin{proof}
	Let $\kappa \in \mathcal{C}^{2r+1}(\Omega)$ and let
	$x \in \mathcal{C}^{2r+2}[0,1]$. Then
	$
	\mathcal{K}(\mathcal{I}-\mathcal{Q}_n)x \in \mathcal{C}^{2r+1}[0,1].
	$
	
	\noindent
	Using the interpolation error estimate \eqref{Eq:05}, we obtain
	\begin{equation}\label{Eq:09}
		\norm{(\mathcal{I}-\mathcal{Q}_n)\mathcal{K}(\mathcal{I}-\mathcal{Q}_n)x}_\infty
		\le C_1\norm{(\mathcal{K}(\mathcal{I}-\mathcal{Q}_n)x)^{(2r+1)}}_\infty h^{2r+1}.
	\end{equation}
	
	\noindent
	Consider 
	\begin{equation*} 
		\mathcal{K}(\mathcal{I} - \mathcal{Q}_n)x(s) =\int_0^1  \kappa(s,t) (\mathcal{I} - \mathcal{Q}_{n}) x(t) \, dt.
	\end{equation*}
	Since $\kappa \in \mathcal{C}^{2r+1}(\Omega)$, differentiation under the integral sign is valid, yielding
	\begin{align*} 
		\left(\mathcal{K}(\mathcal{I} - \mathcal{Q}_n)x\right)^{(2r+1)}(s) &=\int_0^1  \frac{\partial^{2r+1}\kappa}{\partial s^{2r+1}}(s,t) (\mathcal{I} - \mathcal{Q}_{n}) x(t) \, dt \notag \\ 
		&= \sum_{j=1}^{n} \int_{t_{j-1}}^{t_j}  D^{(2r+1,0)}\kappa(s,t) (\mathcal{I} - \mathcal{Q}_{n,j})x(t) \, dt \notag \\
		&=\sum_{j=1}^{n}\int_{t_{j-1}}^{t_j}
		D^{(2r+1,0)}\kappa(s,t)\Phi_j(t)
		[\tau_j^0,\tau_j^1,\ldots,\tau_j^{2r},t]x\,dt.  
	\end{align*}
	For fixed \(s\), we assume that
	\begin{align*}
		g_j^s(t) = D^{(2r+1,0)}\kappa(s,t) \left[\tau_j^0,\tau_j^1,\ldots, \tau_j^{2r},t\right]x, \quad t \in [t_{j-1}, t_j], ~ \text{for} ~ j=1,2,\ldots,n.
	\end{align*}
	Therefore,
	\begin{equation*} 
		\left(\mathcal{K}(\mathcal{I} - \mathcal{Q}_n)x\right)^{(2r+1)}(s)= \sum_{j=1}^{n ~}\int_{t_{j-1}}^{t_j} g_j^s(t) \Phi_j(t) \, dt
	\end{equation*}	
	It is easy to see that, for $s \in [0,1]$ and $t \in [t_{j-1},t_j]$,
	\begin{equation*}
		\left|(g_j^s)'(t) \right| \le 2 C_2 \norm{x}_{2r +2, \infty}.
	\end{equation*}
	Let $c_j$ be the midpoint of $\Delta_j$. By the mean-value form of Taylor's theorem, for each $t\in\Delta_j$ there exists $\xi_t\in\Delta_j$ such that
	\[
	g_j^s(t) = g_j^s(c_j) + (t-c_j)(g_j^s)'(\xi_t).
	\]
	Using the antisymmetry of $\Phi_j$ about $c_j$, we have $\int_{t_{j-1}}^{t_j}\Phi_j(t)\,dt=0$, and hence
	\[
	\int_{t_{j-1}}^{t_j} g_j^s(t)\,\Phi_j(t)\,dt
	= \int_{t_{j-1}}^{t_j} (t-c_j)(g_j^s)'(\xi_t)\,\Phi_j(t)\,dt.
	\]
	Taking absolute values, we obtain
	\[
	\Bigl|\int_{t_{j-1}}^{t_j} g_j^s(t)\,\Phi_j(t)\,dt\Bigr|
	\le 2C_2\norm{x}_{2r+2,\infty}\int_{t_{j-1}}^{t_j}|t-c_j|\,|\Phi_j(t)|\,dt.
	\]
	On $\Delta_j$ we have $|t-c_j|\le h/2$ and, since each factor $(t-\tau_j^i)=O(h)$ on $\Delta_j$,
	\[
	\int_{t_{j-1}}^{t_j}|\Phi_j(t)|\,dt = O(h^{2r+2}).
	\]
	Therefore each subinterval contributes
	\[
	\Bigl|\int_{t_{j-1}}^{t_j} g_j^s(t)\,\Phi_j(t)\,dt\Bigr| = O(h)\cdot O(h^{2r+2})
	= O(h^{2r+3}),
	\]
	uniformly in $s$ and $j$. Summing over $j=1,\dots,n$ (with $n=1/h$) yields
	\[
	\bigl|(\mathcal{K}(\mathcal{I}-\mathcal{Q}_n)x)^{(2r+1)}(s)\bigr|
	\le n\cdot O(h^{2r+3}) = \frac{1}{h}\,O(h^{2r+3}) = O(h^{2r+2}),
	\]
	uniformly in $s$. Hence there exists $C_3 >0$, independent to $h$, such that
	\[
	\norm{(\mathcal{K}(\mathcal{I}-\mathcal{Q}_n)x)^{(2r+1)}}_\infty
	\le C_3 \norm{x}_{2r+2,\infty}\,h^{2r+2}.
	\]
	Finally, substituting this into \eqref{Eq:09} gives
	\[
	\norm{(\mathcal{I}-\mathcal{Q}_n)\mathcal{K}(\mathcal{I}-\mathcal{Q}_n)x}_\infty
	\le C_1C_3 \norm{x}_{2r+2,\infty}\,h^{4r+3},
	\]
	which completes the proof.
\end{proof}

\begin{proposition} \label{Pr:03}
	If the kernel $\kappa \in \mathcal{C}^{2r+2}(\Omega)$ and 
	$x \in \mathcal{C}^{2r+2}[0,1]$, then
	\[
	\norm{\mathcal{K}(\mathcal{I} - \mathcal{Q}_n)\mathcal{K}(\mathcal{I} - \mathcal{Q}_n)x}_\infty 
	\le 4(C_2)^2  \norm{x}_{2r +2, \infty} h^{4r+4}.
	\]
\end{proposition}
\vspace*{-0.2in}
\begin{proof}
	Assume that $\kappa \in \mathcal{C}^{2r+2}(\Omega)$ and let
	$x \in \mathcal{C}^{2r+2}[0,1]$. Then
	$
	\mathcal{K}(\mathcal{I}-\mathcal{Q}_n)x \in \mathcal{C}^{2r+2}[0,1].
	$
	
	\noindent
	Using the Proposition \ref{Pr:01}, we get
	\begin{equation}\label{Eq:10}	
		\norm{\mathcal{K}(\mathcal{I} - \mathcal{Q}_n)\mathcal{K}(\mathcal{I} - \mathcal{Q}_n)x}_\infty \le 2C_2 \norm{\mathcal{K}(\mathcal{I} - \mathcal{Q}_n)x}_{2r+2,\infty} h^{2r+2}.  
	\end{equation}
	To estimate the derivative norm $\norm{\mathcal{K}(\mathcal{I}-\mathcal{Q}_n)x}_{2r+2,\infty}$, 
	note that
	\begin{equation*} 
		\norm{\mathcal{K}(\mathcal{I} - \mathcal{Q}_n)x}_{2r+2,\infty} = \max_{0 \leq i \leq 2r + 2} \norm{\left(\mathcal{K}(\mathcal{I} - \mathcal{Q}_n)x\right)^{(2r+2)}}_\infty.
	\end{equation*}
	Since $\kappa \in \mathcal{C}^{2r+2}(\Omega)$, and the argument used in Proposition~\ref{Pr:02} applies. 
	Following the same reasoning, one obtains
	\begin{equation*}
		\left|\left(\mathcal{K}(\mathcal{I} - \mathcal{Q}_n)x\right)^{(2r+2)}(s)\right| \le  2 C_2 \norm{x}_{2r +2, \infty} h^{2r + 2}.
	\end{equation*}
	From estimate \eqref{Eq:10}, we have
	\begin{equation*}	
		\norm{\mathcal{K}(\mathcal{I} - \mathcal{Q}_n)\mathcal{K}(\mathcal{I} - \mathcal{Q}_n)x}_\infty \le 4(C_2)^2  \norm{x}_{2r +2, \infty} h^{4r+4}.
	\end{equation*}
	This completes the proof.
\end{proof}
Propositions~\ref{Pr:01}-\ref{Pr:03} provide the fundamental approximation estimates for the interpolatory projection operator $\mathcal{Q}_n$, which will be instrumental in deriving the convergence rates for the corresponding eigenvalue and eigenfunction approximations.
\begin{lemma} \label{Lem:04}
	Let \(\mathcal{K}\) be the linear integral operator defined by \eqref{Eq:01} with a smooth kernel $\kappa \in\mathcal{C}^{2r+1}(\Omega)$ and continuously differentiable with respect to $t$. Let \(\mathcal{Q}_n\) denote the interpolatory projection operator onto the approximation space \(\mathcal{X}_n\) as defined in \eqref{Eq:04}. Then
	\begin{eqnarray*}
		\norm{(\mathcal{K} - \mathcal{Q}_n \mathcal{K}) \mathcal{K} |_{\mathcal{R}(E)} } &=& O\left(h^{2r + 1}\right), \\
		\norm{ \mathcal{K} (\mathcal{K} - \mathcal{Q}_n \mathcal{K}) \mathcal{K} |_{\mathcal{R}(E)} } &=& O\left(h^{2r + 2}\right).
	\end{eqnarray*}
\end{lemma}
\vspace*{-0.2in}
\begin{proof}
	By definition of the operator norm,
	\[
	\norm{(\mathcal{K} - \mathcal{Q}_n \mathcal{K}) \mathcal{K} 
		\big|_{\mathcal{R}(E)}}
	=
	\sup_{\psi \in \mathcal{R}(E)}
	\Big\{
	\norm{(\mathcal{K} - \mathcal{Q}_n \mathcal{K}) \mathcal{K} \psi }_\infty
	:\, \norm{\psi}_\infty = 1
	\Big\}.
	\]	
	Let $\psi \in \mathcal{R}(E)$ and set $y = \mathcal{K}\psi$. 
	Since $\kappa \in \mathcal{C}^{2r+1}(\Omega)$, differentiation under 
	the integral sign is valid up to order $2r+1$. Hence,
	$\mathcal{K} : \mathcal{C}[0,1] \to \mathcal{C}^{2r+1}[0,1]$, 
	so that $y \in \mathcal{C}^{2r+1}[0,1]$ and consequently 
	$\mathcal{K}y \in \mathcal{C}^{2r+2}[0,1]$.
	Then using estimate~\eqref{Eq:05},
	\[
	\norm{(\mathcal{I} - \mathcal{Q}_n)\mathcal{K}y}_\infty
	\le
	C_1 \norm{(\mathcal{K}y)^{(2r+1)}}_\infty\, h^{2r+2}.
	\]
	From inequality~\eqref{Eq:06},
	\[
	\norm{(\mathcal{K}y)^{(2r+1)}}_\infty
	\le
	(C_2)^2 \norm{\psi}_\infty.
	\]
	Hence,
	\[
	\norm{(\mathcal{K} - \mathcal{Q}_n \mathcal{K}) \mathcal{K} \psi }_\infty
	\le
	(C_2)^2\, \norm{\psi}_\infty\, h^{2r+2},
	\]
	and therefore
	\[
	\norm{(\mathcal{K} - \mathcal{Q}_n \mathcal{K}) \mathcal{K} 
		\big|_{\mathcal{R}(E)}}
	=
	O\left(h^{2r+2}\right).
	\]
	
	Next, we estimate
	\[
	\norm{\mathcal{K}(\mathcal{K} - \mathcal{Q}_n \mathcal{K}) 
		\mathcal{K} \big|_{\mathcal{R}(E)}}
	=
	\sup_{\psi \in \mathcal{R}(E)}
	\Big\{
	\norm{\mathcal{K}(\mathcal{I} - \mathcal{Q}_n)\mathcal{K}y}_\infty
	:\, \norm{\psi}_\infty = 1
	\Big\}.
	\]
	Applying Proposition~\ref{Pr:01},
	\[
	\norm{\mathcal{K}(\mathcal{I} - \mathcal{Q}_n)\mathcal{K}y}_\infty
	\le
	2C_2 \norm{\mathcal{K}y}_{2r+2,\infty}\, h^{4r+4}.
	\]
	Note that 
	\begin{equation*}
		\norm{\mathcal{K}y}_{2r+2, \infty} = \max_{0 \leq j \leq2r +2} \norm{ \left(\mathcal{K}y\right)^{(j)}}_\infty.
	\end{equation*}
	From inequality \eqref{Eq:06}, it follows that the bound $\norm{(\mathcal{K}y)^{(j)}}_\infty$ is independent of $h$.
	Thus,
	\[
	\norm{\mathcal{K}(\mathcal{K} - \mathcal{Q}_n \mathcal{K}) 
		\mathcal{K} \big|_{\mathcal{R}(E)}}
	=
	O\left(h^{4r+4}\right).
	\]
	This completes the proof.
\end{proof} 

\begin{lemma} \label{Lem:05}
	Let \(\mathcal{K}\) be the linear integral operator defined in \eqref{Eq:01} with a smooth kernel $\kappa \in \mathcal{C}^{2r+2}[0,1]$. Let \(\mathcal{Q}_n\) denote the interpolatory projection operator onto the approximation space \(\mathcal{X}_n\), as given by \eqref{Eq:04}. Then 
	\begin{eqnarray*}
		\norm{(\mathcal{K} - \mathcal{K}_n^M) \mathcal{K} |_{\mathcal{R}(E)}} 
		&=& O\left(h^{4r+3}\right), \\
		\norm{\mathcal{K} (\mathcal{K} - \mathcal{K}_n^M) \mathcal{K} |_{\mathcal{R}(E)}} 
		&=& O\left(h^{4r+4}\right),
	\end{eqnarray*}
	where \(\mathcal{K}_n^M = \mathcal{Q}_n \mathcal{K} + \mathcal{K} \mathcal{Q}_n - \mathcal{Q}_n \mathcal{K} \mathcal{Q}_n\).
\end{lemma}
\vspace*{-0.2in}
\begin{proof}
	Let $\psi \in \mathcal{R}(E)$. Since $\mathcal{R}(E)$ is invariant under 
	$\mathcal{K}$, it follows that $\mathcal{K}\psi \in \mathcal{R}(E)$. 
	By definition of the operator norm restricted to $\mathcal{R}(E)$,
	\[
	\norm{(\mathcal{K} - \mathcal{K}_n^M)\mathcal{K} 
		\big|_{\mathcal{R}(E)}}
	=
	\sup_{\psi \in \mathcal{R}(E)}
	\Big\{
	\norm{(\mathcal{I}-\mathcal{Q}_n)\mathcal{K}
		(\mathcal{I}-\mathcal{Q}_n)\mathcal{K}\psi}_\infty
	:\, \norm{\psi}_\infty = 1
	\Big\}.
	\]
	Assume $\psi \in \mathcal{C}^{2r+2}[0,1]$ and  $\kappa \in \mathcal{C}^{2r+2}(\Omega)$. Then $\mathcal{K}\psi \in \mathcal{C}^{2r+2}[0,1]$. 
	Applying Proposition~\ref{Pr:02}, we obtain
	\[
	\norm{(\mathcal{I}-\mathcal{Q}_n)\mathcal{K}
		(\mathcal{I}-\mathcal{Q}_n)\mathcal{K}\psi}_\infty
	\le
	C_1 C_3\norm{\mathcal{K}\psi}_{2r+2,\infty} \, h^{4r+3}.
	\]
	Moreover,
	\[
	\norm{\mathcal{K}\psi}_{2r+2,\infty}
	=
	\max_{0 \le j \le 2r+2}
	\norm{(\mathcal{K}\psi)^{(j)}}_\infty
	\le
	C_2 \norm{\psi}_\infty,
	\]
	so that
	\[
	\norm{(\mathcal{I}-\mathcal{Q}_n)\mathcal{K}
		(\mathcal{I}-\mathcal{Q}_n)\mathcal{K}\psi}_\infty
	\le
	C_1 C_2 C_3\norm{\psi}_\infty \, h^{4r+3}.
	\]
	Taking $\norm{\psi}_\infty = 1$, we conclude
	\[
	\norm{(\mathcal{K} - \mathcal{K}_n^M)\mathcal{K} 
		\big|_{\mathcal{R}(E)}}
	=
	O\left(h^{4r+3}\right).
	\]
	
	Next, consider
	\[
	\norm{\mathcal{K}(\mathcal{K} - \mathcal{K}_n^M)\mathcal{K}
		\big|_{\mathcal{R}(E)}}
	=
	\sup_{\psi \in \mathcal{R}(E)}
	\Big\{
	\norm{\mathcal{K}(\mathcal{I}-\mathcal{Q}_n)\mathcal{K}
		(\mathcal{I}-\mathcal{Q}_n)\mathcal{K}\psi}_\infty
	:\, \norm{\psi}_\infty = 1\Big\}.
	\]
	Since $\mathcal{K}\psi \in \mathcal{C}^{2r+2}[0,1]$, Proposition~\ref{Pr:03} yields
	\[
	\norm{\mathcal{K}(\mathcal{I}-\mathcal{Q}_n)\mathcal{K}
		(\mathcal{I}-\mathcal{Q}_n)\mathcal{K}\psi}_\infty
	\le
	4(C_2)^2 \norm{\mathcal{K}\psi}_{2r+2,\infty} \, h^{4r+4}.
	\]
	Using again the bound 
	$\norm{\mathcal{K}\psi}_{2r+2,\infty} \le C_2 \norm{\psi}_\infty$, 
	we obtain
	\[
	\norm{\mathcal{K}(\mathcal{K} - \mathcal{K}_n^M)\mathcal{K}
		\big|_{\mathcal{R}(E)}}
	\le
	4(C_2)^3 \norm{\psi}_\infty \, h^{4r+4}.
	\]
	Hence, for $\norm{\psi}_\infty = 1$,
	\[
	\norm{\mathcal{K}(\mathcal{K} - \mathcal{K}_n^M)\mathcal{K}
		\big|_{\mathcal{R}(E)}}
	=
	O\left(h^{4r+4}\right).
	\]
	This completes the proof.
\end{proof}

By combining Lemma~\ref{Lem:04} with Theorem~\ref{Thm:01},  we obtain the following convergence rates for the approximation of eigenvalues and spectral subspaces by the collocation and iterated collocation methods.
\begin{theorem} \label{thm1}
	Let \(\psi_n^C \in \mathcal{R}(E_n)\) denote a collocation eigenfunction, and let \(\psi_n^S\) be the corresponding iterated collocation eigenfunction. Suppose \(\lambda_n^C\) is the approximate simple eigenvalue of \(\mathcal{Q}_n\mathcal{K}\). Then, for sufficiently large \(n\), the following error bounds hold:
	\[	
	\norm{\psi_n^C - E\psi_n^C  }_{\infty} = O\left(h^{2r +1}\right), \quad 	\norm{\psi_n^S - E\psi_n^S  }_{\infty} = O\left(h^{2r +2}\right),
	\]
	and
	\[	
	|\lambda - \lambda_n^C| = O\left(h^{2r +2}\right).
	\]
\end{theorem}

Similarly, by combining Lemma~\ref{Lem:05} with Theorem~\ref{Thm:02}, we derive the following convergence estimates for the modified collocation and iterated modified collocation methods.
\begin{theorem} \label{thm2}
	Let \(\psi_n^M \in \mathcal{R}(E_n^M)\) be a modified collocation eigenfunction, and let \(\widetilde{\psi}_n^M\) denote the corresponding iterated modified collocation eigenfunction. Suppose \(\lambda_n^M\) is the approximate simple eigenvalue of \(\mathcal{K}_n^M\) obtained from the modified method. Then, for all sufficiently large \(n\), the following convergence orders are satisfied:
	\[	
	\norm{\psi_n^M - E\psi_n^M }_{\infty}  = O\left(h^{4r+3}\right), \quad 	\norm{\widetilde{\psi}_n^M - E\widetilde{\psi}_n^M }_{\infty} = O\left(h^{4r+4}\right),
	\]
	and
	\[
	|\lambda - \lambda_n^M| = O\left(h^{4r+4}\right).
	\]
\end{theorem}

\noindent
The above results show that the standard collocation method produces eigenfunction and eigenvalue approximations of orders $O(h^{2r+1})$ and $O(h^{2r+2})$, respectively, while the iterated collocation method improves the eigenfunction accuracy by one additional order. In contrast, the modified collocation method significantly enhances the convergence behavior, yielding eigenvalue approximations of order $O(h^{4r+4})$, with the iterated modified method attaining the highest accuracy. These findings confirm that modified collocation  and iteration are highly effective strategies for achieving high-accuracy approximations of eigenelements, even when simple equidistant collocation points are employed on each subinterval of a uniform partition over an even-degree piecewise polynomial approximation space.


\section{Numerical Results}

In this section, we present two numerical experiments to illustrate the theoretical convergence results established in the previous section. In both examples, we approximate the largest eigenvalue and associated eigenfunction using the classical collocation and the modified collocation methods, together with their iterated versions.

We consider the integral operator $\mathcal{K}$ with smooth kernel $\kappa$ as
\[
(\mathcal{K}x)(s) = \int_{0}^{1} \kappa(s,t)\, x(t) \, dt, \quad s \in [0, 1],~ x \in \mathcal{C}[0,1].
\]
For the discretization, we consider the space $\mathcal{X}_n$ of piecewise constant functions defined on the uniform partition
\[
0 = t_0 < t_1 < \cdots < t_n = 1,
\quad
t_j = \frac{j}{n}, \quad j=0,1,\dots,n.
\]
The interpolatory projection $\mathcal{Q}_n : \mathcal{C}[0,1] \to \mathcal{X}_n$
is defined by \[
\mathcal{Q}_nx(\tau_j) = x(\tau_j), \quad j = 1, 2, \ldots, n,
\]
where the collocation nodes are chosen as the midpoints
\[
\tau_j = \frac{2j - 1}{2n}, \quad j = 1, 2, \ldots, n.
\]

We approximate the eigenvalue of $\mathcal{K}$ with largest modulus 
and its corresponding eigenfunction as follows: 

\noindent
In all computations, the observed rate of convergence is estimated from two consecutive mesh refinements according to
\[
\delta = \frac{\log(e_n/e_{2n})}{\log 2},
\]
where $e_n$ represents the error obtained with $n$ uniform subintervals.

\vspace*{0.1in}
\noindent
\textbf{Example 1:  $\kappa(s,t)=e^{st}, ~ s,t \in[0,1]$}

\noindent
We first consider the integral operator
\[
(\mathcal{K}x)(s)
=
\int_0^1 e^{st} x(t)\,dt,
\quad s\in[0,1].
\]
The kernel $\kappa(s,t)=e^{st}$ is smooth on $\Omega$. It is known that all eigenvalues of $\mathcal{K}$ are simple, and the largest eigenvalue is $\lambda = 1.3530301647457353.$

\begin{table}[h!]
	\small
	\centering
	\setlength{\tabcolsep}{10pt}
	\renewcommand{\arraystretch}{1.2}
	\caption{Eigenvalues approximation errors for Example~$1$  in variants of Collocation methods}
	\label{table:01}
	\begin{tabular}{|c|c|c|c|c|}
		\hline
		$n$ & $|\lambda - \lambda_n^C|$ & $\delta_{C}$ & $|\lambda - \lambda_n^M|$ & $\delta_{MC}$  \\[3pt] 
		\hline
		2   & $2.08 \times 10^{-2}$ &            & $6.25 \times 10^{-4}$ &          \\[3pt] 
		4   & $5.47 \times 10^{-3}$ & $1.93$     & $4.20 \times 10^{-5}$ & $3.90$     \\[3pt] 
		8   & $1.39 \times 10^{-3}$ & $1.98$     & $2.67 \times 10^{-6}$ & $3.97$     \\[3pt] 
		16  & $3.48 \times 10^{-4}$ & $2.00$     & $1.68 \times 10^{-7}$ & $3.99$     \\[3pt] 
		32  & $8.70 \times 10^{-5}$ & $2.00$     & $1.05 \times 10^{-8}$ & $4.00$     \\[3pt] 
		64  & $2.18 \times 10^{-5}$ & $2.00$     & $6.56 \times 10^{-10}$ & $4.00$     \\[3pt]
		128  & $5.44 \times 10^{-6}$ & $2.00$     & $4.10 \times 10^{-11}$ & $4.00$     \\[3pt] 
		\hline
	\end{tabular}
\end{table}

\begin{table}[h!]
	\small
	\centering
	\setlength{\tabcolsep}{2.3pt}
	\renewcommand{\arraystretch}{1.2}
	\caption{Eigenfunction errors in variants of Collocation methods for Example~$1$}
	\label{table:02}
	\begin{tabular}{|c|cc|cc|cc|cc|}
		\hline
		$n$ & $\norm{\psi_n^C - E \psi_n^C }_\infty$ & $\delta_C$ 
		& $\norm{\psi_n^S - E \psi_n^S }_\infty$ & $\delta_{IC}$ 
		& $\norm{\psi_n^M - E \psi_n^M }_\infty$ & $\delta_{MC}$ 
		& $\norm{\widetilde{\psi}_n^M - E \widetilde{\psi}_n^M }_\infty$ & $\delta_{IMC}$ \\[3.5pt]
		\hline
		2   & $2.62\times10^{-1}$ &      & $7.34\times10^{-3}$ &      & $7.51\times10^{-3}$ &      & $2.33\times10^{-4}$ &      \\[3pt] 
		4   & $1.43\times10^{-1}$ & 0.87 & $1.87\times10^{-3}$ & 1.97 & $1.18\times10^{-3}$ & 2.67 & $1.56\times10^{-5}$ & 3.90 \\[3pt] 
		8   & $7.49\times10^{-2}$ & 0.93 & $4.71\times10^{-4}$ & 1.99 & $1.64\times10^{-4}$ & 2.84 & $9.93\times10^{-7}$ & 3.97 \\[3pt] 
		16  & $3.83\times10^{-2}$ & 0.97 & $1.18\times10^{-4}$ & 2.00 & $2.16\times10^{-5}$ & 2.92 & $6.24\times10^{-8}$ & 3.99 \\[3pt] 
		32  & $1.92\times10^{-2}$ & 0.99 & $2.95\times10^{-5}$ & 2.00 & $2.75\times10^{-6}$ & 2.97 & $3.90\times10^{-9}$ & 4.00 \\[3pt] 
		64  & $9.67\times10^{-3}$ & 0.99 & $7.38\times10^{-6}$ & 2.00 & $3.49\times10^{-7}$ & 2.98 & $2.43\times10^{-10}$ & 4.00 \\[3pt] 
		128 & $4.78\times10^{-3}$ & 1.02 & $1.84\times10^{-6}$ & 2.00 & $4.33\times10^{-8}$ & 3.01 & $1.48\times10^{-11}$ & 4.04 \\[3pt] 
		\hline
	\end{tabular}
\end{table}

\bigskip

\medskip
\noindent
\textbf{Example 2:  $\kappa(s,t)=\cos(\pi s)\cos(\pi t), ~ s,t \in[0,1]$}

\noindent
As a second test problem, we consider the integral operator
\[
(\mathcal{K}x)(s)
=
\int_0^1 \cos(\pi s)\cos(\pi t) x(t)\,dt,
\quad s\in[0,1].
\]
with separable kernel $\kappa(s,t)=\cos(\pi s)\cos(\pi t),~s,t\in[0,1],$ smooth on $\Omega$. In this case, the integral operator is of rank one and admits the unique nonzero simple eigenvalue as $\lambda = \frac{1}{2}$ and associated eigenfunction as $\psi(s) = \cos(\pi s).$ All other eigenvalues are equal to zero.

\begin{table}[h!]
	\small
	\centering
	\setlength{\tabcolsep}{10pt}
	\renewcommand{\arraystretch}{1.2}
	\caption{Eigenvalues approximation errors for Example~$2$  in variants of Collocation methods}
	\label{table:03}
	\begin{tabular}{|c|c|c|c|c|}
		\hline
		$n$ & $|\lambda - \lambda_n^C|$ & $\delta_{C}$ & $|\lambda - \lambda_n^M|$ & $\delta_{MC}$ \\[3pt]
		\hline
		2   & $4.98 \times 10^{-2}$ &            & $4.68 \times 10^{-3}$ &          \\[3pt]
		4   & $1.28 \times 10^{-2}$ & $1.97$     & $3.20 \times 10^{-4}$ & $3.87$     \\[3pt]
		8   & $3.21 \times 10^{-3}$ & $1.99$     & $2.05 \times 10^{-5}$ & $3.96$     \\[3pt]
		16  & $8.03 \times 10^{-4}$ & $2.00$     & $1.29 \times 10^{-6}$ & $3.99$     \\[3pt]
		32  & $2.01 \times 10^{-4}$ & $2.00$     & $8.06 \times 10^{-8}$ & $4.00$     \\[3pt]
		64  & $5.02 \times 10^{-5}$ & $2.00$     & $5.04 \times 10^{-9}$ & $4.00$     \\[3pt]
		\hline
	\end{tabular}
\end{table}

\begin{table}[h!]
	\small
	\centering
	\setlength{\tabcolsep}{2.3pt}
	\renewcommand{\arraystretch}{1.2}
	\caption{Eigenfunction errors in variants of Collocation methods for Example~$2$}
	\label{table:04}
	\begin{tabular}{|c|cc|cc|cc|cc|}
		\hline
		$n$ & $\norm{\psi_n^C - E \psi_n^C }_\infty$ & $\delta_C$ 
		& $\norm{\psi_n^S - E \psi_n^S }_\infty$ & $\delta_{IC}$ 
		& $\norm{\psi_n^M - E \psi_n^M }_\infty$ & $\delta_{MC}$ 
		& $\norm{\widetilde{\psi}_n^M - E \widetilde{\psi}_n^M }_\infty$ & $\delta_{IMC}$ \\[3.5pt]
		\hline
		2   & $6.43\times10^{-1}$ &      & $9.07\times10^{-2}$ &      & $3.34\times10^{-2}$ &      & $4.71\times10^{-3}$ &      \\
		4   & $3.73\times10^{-1}$ & 0.79 & $2.49\times10^{-2}$ & 1.87 & $4.82\times10^{-3}$ & 2.79 & $3.21\times10^{-4}$ & 3.88   \\[3pt]
		8   & $1.94\times10^{-1}$ & 0.94 & $6.37\times10^{-3}$ & 1.96 & $6.24\times10^{-4}$ & 2.95 & $2.05\times10^{-5}$ & 3.97   \\[3pt]
		16  & $9.79\times10^{-2}$ & 0.99 & $1.60\times10^{-3}$ & 1.99 & $7.86\times10^{-5}$ & 2.99 & $1.29\times10^{-6}$ & 3.99   \\[3pt]
		32  & $4.90\times10^{-2}$ & 1.00 & $4.01\times10^{-4}$ & 2.00 & $9.85\times10^{-6}$ & 3.00 & $8.06\times10^{-8}$ & 4.00   \\[3pt]
		64  & $2.45\times10^{-2}$ & 1.00 & $1.00\times10^{-4}$ & 2.00 & $1.23\times10^{-6}$ & 3.00 & $5.04\times10^{-9}$ & 4.00   \\[3pt]
		\hline
	\end{tabular}
\end{table}

The numerical experiments in both examples clearly validate the theoretical convergence results. For the eigenvalue approximations (Tables~\ref{table:01} and \ref{table:03}), 
the classical collocation method exhibits second-order convergence, 
whereas the modified collocation method achieves fourth-order convergence. 
For the eigenfunction approximations (Tables~\ref{table:02} and \ref{table:04}), 
the classical method converges with first-order accuracy and its iterated version improves the rate to second order. Similarly, the modified collocation method attains third-order convergence, while its iterated version reaches fourth-order accuracy. Overall, the numerical results demonstrate that the modified collocation method provides a higher rate of convergence than the classical collocation method for both eigenvalue and eigenfunction approximations. Furthermore, the iterated versions improve the accuracy of the corresponding eigenfunction approximations without altering the theoretical convergence orders. These observations are consistent with the analytical error estimates established earlier.

\section{Conclusion}

In this paper, we analyze the convergence of approximate solutions for eigenvalue problems associated with integral operators having smooth kernels. Using interpolatory projection-based methods, we show that the modified colloacation method yields faster convergence of approximate eigenvalues compared to the classical collocation schemes. 
Furthermore, the iterated variants of these methods exhibit superconvergent behaviour, leading to improved approximations of eigenfunctions and spectral subspaces. 
Such improvements are particularly advantageous in practical computations, where higher accuracy is desired without increasing the dimension of the resulting linear system.

Our results extend existing work by providing explicit convergence rates and clarifying the role of iteration in enhancing spectral subspace approximations, while allowing the collocation points to be chosen independently of the zeros of special orthogonal polynomials. Possible directions for future research involve broadening the present framework to cover integral operators with non-smooth, singular or weakly singular kernels and developing refined asymptotic expansions for eigenvalue approximations.


\section*{Declarations}

\subsection*{Conflict of Interest}
The author has no conflict of interest to declare.

\subsection*{Funding}
No funding was received to assist with the preparation of this manuscript.

\subsection*{Acknowledgement}

The author gratefully acknowledges his PhD advisor, Dr. Gobinda Rakshit, for his guidance and support. The author also acknowledges RGIPT Jais and IIT Goa, India, for their institutional support during the course of this research.



\begin{thebibliography}{99}
	
	\bibitem{Ahues} 
	M. Ahues, A. Largillier, and B. V. Limaye, \textit{Spectral computations for bounded operators}, Chapman and Hall/CRC, London, 2001.
	
	\bibitem{CA-PS-DS-MT}
	C. Allouch, P. Sablonniere, D. Sbibih, and M. Tahrichi, Superconvergent Nystr\"om and degenerate kernel methods for eigenvalue problems. \textit{Applied Mathematics and Computation}, \textbf{217}:20 (2011), 7851--7866.
	\url{https://doi.org/10.1016/j.amc.2011.01.098}
	
	\bibitem{Atkinson1967}
	K. E. Atkinson, The numerical solution of the eigenvalue problem for compact integral operators. \textit{Transactions of the American Mathematical Society}, \textbf{129} (1967), 458--465.
	\url{https://doi.org/10.2307/1994601}
	
	\bibitem{Atkinson1975}
	K. E. Atkinson, Convergence rates for approximate eigenvalues of compact integral operators. \textit{SIAM Journal on Numerical Analysis}, \textbf{12}:2 (1975), 213--222.
	
	\bibitem{KEA0}
	K. E. Atkinson, \textit{The numerical solution of integral equations of the second kind}, Cambridge University Press, Cambridge, 1997.
	\url{https://doi.org/10.1137/0712020}
	
	\bibitem{KEA-IG-IS}
	K. E. Atkinson, I. Graham, and I. Sloan, Piecewise continuous collocation for integral equations. \textit{SIAM Journal on Numerical Analysis}, \textbf{20}:1 (1983), 172--186.
	\url{https://doi.org/10.1137/0720012}
	
	\bibitem{Babuska1987}
	I. Babu\v{s}ka and J. E. Osborn, Estimates for the errors in eigenvalue and eigenfunction approximation by Galerkin methods, with particular attention to the case of multiple eigenvalues. \textit{SIAM Journal on Numerical Analysis}, \textbf{24}:6 (1987), 1249--1276.
	\url{https://doi.org/10.1137/0724082}
	
	\bibitem{Babuska1989}
	I. Babu\v{s}ka and J. E. Osborn, Finite element-Galerkin approximation of the eigenvalues and eigenfunctions of selfadjoint problems. \textit{Mathematics of Computation}, \textbf{52}:186 (1989), 275--297.
	\url{https://doi.org/10.2307/2008468}
	
	\bibitem{Baker0} 
	C. T. H. Baker, \textit{The numerical treatment of integral equations}, Oxford University Press, Oxford, 1977.
	
	\bibitem{HB-CA-ZEA-KK}
	H. Bouda, C. Allouch, Z. El Allali, and K. Kant, Numerical solution of eigenvalue problems for a compact integral operator with Green's kernels. \textit{Advances in Operator Theory}, \textbf{9}:51 (2024).  
	\url{https://doi.org/10.1007/s43036-024-00352-7}
	
	\bibitem{Bramble1973}
	J. H. Bramble and J. E. Osborn, Rate of convergence estimates for nonselfadjoint eigenvalue approximations. \textit{Mathematics of Computation}, \textbf{27}:123 (1973), 525--549.
	
	\bibitem{Cha-Leb2} 
	F. Chatelin and R. Lebbar, Superconvergence results for the iterated projection method applied to a Fredholm integral equation of the second kind and the corresponding eigenvalue problem. \textit{Journal of Integral Equations}, \textbf{6}:1 (1984), 71--91. 
	\url{http://www.jstor.org/stable/26164159}
	
	\bibitem{Chatelin} 
	F. Chatelin, \textit{Spectral approximation of linear operators}, Academic Press, New York, 1983.
	
	\bibitem{Chen1997}
	M. Chen, Z. Chen, and G. Chen, \textit{Approximate Solutions of Operator Equation}, World Scientific, Singapore, 1997.
	
	\bibitem{Gnan1}
	N. Gnaneshwar, A degenerate kernel method for eigenvalue problems of compact integral operators. \textit{Advances in Computational Mathematics}, \textbf{27} (2007), 339--354.
	\url{https://doi.org/10.1007/s10444-005-9005-9}
	
	\bibitem{Hild}
	F. B. Hilderbrand, \textit{Introduction to Numerical Analysis}, 2nd edition, McGraw-Hill, 1974.
	
	\bibitem{Kato1976}
	T. Kato, \textit{Perturbation Theory for Linear Operators}, Springer-Verlag, Berlin, 1976.
	
	\bibitem{RPK5}
	R. P. Kulkarni, A new superconvergent projection method for approximate solutions of eigenvalue problems. \textit{Numerical Functional Analysis and Optimization}, \textbf{24}:1--2 (2003), 75--84.
	\url{https://doi.org/10.1081/NFA-120020246}
	
	\bibitem{RPK6}
	R. P. Kulkarni, A new superconvergent collocation method for eigenvalue problems. \textit{Mathematics of Computation}, \textbf{75}:254 (2006), 847--857.
	\url{https://doi.org/10.1090/S0025-5718-06-01871-0}
	
	
	
	\bibitem{Osborn}
	J. E. Osborn, Spectral Approximation for compact operators. \textit{Mathematics of Computation}, \textbf{29}:131 (1975), 712--725. 
	\url{https://doi.org/10.2307/2005282}
	
	\bibitem{Pallav2002}
	R. Pallav and A. Pedas, Quadratic spline collocation method for weakly singular integral equations and corresponding eigenvalue problem. \textit{Mathematical Modelling and Analysis}, \textbf{7}:2 (2002), 285--296.
	\url{https://doi.org/10.3846/13926292.2002.9637200}
	
	
	\bibitem{Sloan0}
	I. H. Sloan, Superconvergence. In \textit{Numerical Solution of Integral Equations}, Edited by M. Golberg, Plenum Press, New York (1990), 35--70.
	
	\bibitem{Sloan1} 
	I. H. Sloan, Iterated Galerkin method for eigenvalue problems. \textit{SIAM Journal on Numerical Analysis}, \textbf{13}:5 (1976), 753--760.
	\url{http://www.jstor.org/stable/2156103}
	
\end{thebibliography}
\end{document}